\documentclass[12pt,a4paper, final, twoside]{article}

\usepackage{amsmath}
\usepackage{fancyhdr}
\usepackage{amsthm}
\usepackage{amsfonts}
\usepackage{amssymb}
\usepackage{amscd}
\usepackage{latexsym}
\usepackage{amsthm}
\usepackage{graphicx}
\usepackage{graphics}
\usepackage{afterpage}
\usepackage[colorlinks=true, urlcolor=blue,  linkcolor=black, citecolor=black]{hyperref}
\usepackage{color}
\theoremstyle{plain} 
\newtheorem{thm}{Theorem}[section]

\newtheorem{lem}[thm]{Lemma}

\theoremstyle{proposition}
\newtheorem{pr}{Proposition}[section]

\theoremstyle{definition}
\newtheorem{defi}{Definition}[section]
\theoremstyle{remark}

\numberwithin{equation}{section}
\begin{document}
\hyphenpenalty=100000

\begin{center}
{\Large {\textbf{\\ Periodic solutions of integro-differential equations in Banach space having Fourier type}}}\\[5mm]
{\large \textbf{{Bahloul Rachid}$^\mathrm{{\bf \color{red}{1}}}$\footnote{\emph{{1} : E-mail address : bahloul33r@hotmail.com}} } \\[1mm]
{\footnotesize $^\mathrm{}$ 
      }\\[3mm]}

\end{center}

\begin{flushleft}\footnotesize \it \textbf{$^\mathrm{{\bf \color{red}{1}}}$ Department of Mathematics, Faculty of Sciences and Technology, {\bf Fez}, Morocco.
}\\[3mm]
\end{flushleft}

\begin{center}\textbf {ABSTRACT}\end{center} 
{\footnotesize {{The aim of this work is to study the existence of a periodic solutions of  integro-differential equations $\frac{d}{dt}[x(t) - L(x_{t})]= A[x(t)- L(x_{t})]+G(x_{t})+ \int_{- \infty}^{t}a(t-s)x(s)ds+f(t)$, ($0 \leq t \leq 2\pi$) with the periodic condition $x(0) = x(2\pi)$, where $a \in L^{1}(\mathbb{R}_{+})$ . Our approach is based on the  M-boundedness
of linear operators, Fourier type, $B^{s}_{p, q}$-multipliers and Besov spaces.}}}\\
\footnotesize{{\textbf{Keywords:}}  integro-differential equations, Fourier type, $B^{s}_{p, q}$-multipliers.}\\[1mm]

\afterpage{
\fancyhead{} \fancyfoot{}
\fancyfoot[R]{\footnotesize\thepage}
\fancyhead[R]{\scriptsize \it{ }
 }}

\section{Introduction}
The aim of this paper is to study the existence and of solutions
for some neutral functional integro-differential equations with delay by using methods of maximal regularity in Besov spaces.
Motivated by the fact that  neutral functional integro-differential equations  with finite delay arise in many areas of applied mathematics, this type of
equations has received much attention in recent years. In particular, the problem of existence of periodic solutions, has been considered by several authors. We refer the readers to papers [\cite{1}, \cite{7}, \cite{9}, \cite{14}] and the references listed therein for information on this subject. One of the most important tools to
prove maximal regularity is the theory of Fourier multipliers. They play an important
role in the analysis of parabolic problems. In recent years it has become apparent
that one needs not only the classical theorems but also vector-valued extensions with
operator-valued multiplier functions or symbols. These extensions allow to treat certain problems for evolution equations with partial differential operators in an elegant
and efficient manner in analogy to ordinary differential equations. For some recent papers on the subjet, we refer to Arendt-Bu \cite{1}, Bu et al \cite{11} and Hernan et al \cite{12}.\\
We characterize the existence of periodic solutions for the following integro-differential equations in vector-valued spaces and Besov. 
In the case of vector-valued space, our results involve UMD spaces, the concept of R-boundedness and a condition on the resolvent operator. We remark that many of the most powerful modern
theorems are valid in UMD spaces, i.e., Banach spaces in which martingale are unconditional differences. The probabilistic definition of UMD spaces turns out to be
equivalent to the $L^{p}$-boundedness of the Hilbert transform, a transformation which
is, in a sense, the typical representative example of a multiplier operator. On the
other hand the notion of M-boundedness has played an important role in the functional  approach to partial differential equations.\\
In the case of Besov  spaces, our results involve only boundedness of the resolvent.

In this work, we study the existence of periodic solutions for the following integro-differential equations 

\begin{equation}\label{e1}
\displaystyle{\frac{d}{dt}[x(t) - L(x_{t})]= A[x(t)- L(x_{t})]+G(x_{t})+ \int_{- \infty}^{t}a(t-s)x(s)ds+f(t)}
\end{equation}
where  $A : D(A) \subseteq  X \rightarrow X$ is a linear closed operator on  Banach space ($X, \left\|.\right\|$) and
$f \in L^{p}(\mathbb{T}, X)$ for all $p \geq 1$. For $r_{2 \pi} := 2\pi N$  ( some $N \in \mathbb{N}$) $L$ and $G$ are in $B(L^{p}([- r_{2 \pi} ,0],\ \ X);\ \ X)$  is the space of all bounded linear operators  and  $x_{t}$ is an element of $L^{p}([- r_{2 \pi}\ \ ,0],\ \ X)$  which is defined as follows $$x_{t}(\theta) = x(t+\theta)\;\; \text{for}\;\; \theta \in[- r_{2 \pi},\ \ 0].$$
In \cite{1}, Arendt gave necessary and sufficient conditions for the existence of periodic solutions of  the following evolution equation.
$$\displaystyle{\frac{d}{dt}x(t)= Ax(t)+f(t)}\;\; \text{for}\;\; t \in \mathbb{R},$$
where $A$ is a closed linear operator on an UMD-space $Y$. \\
In \cite{12}, Hernan et al, studied the existence of periodic solution for the class of linear abstract neutral functional differential equation described in the following form:

\[\frac{d}{dt}[x(t) - Bx(t - r)]= Ax(t)+G(x_{t})+f(t)\;\;\text{ for}\;\; t \in \mathbb{R}\]
where $A: D(A) \rightarrow X$ and $B : D(B) \rightarrow X$ are closed linear operator such that $D(A) \subset D(B)$ and $G \in B(L^{p}([- 2\pi,\ \ 0],\ \ X);\ \ X)$.\\
In {\bf \color{green}{\cite{8}}}, Sylvain Koumla, Khalil Ezzinbi,  Rachid Bahloul established mild solutions for some partial functional integrodifferential equations with finite delay in Fréchet spaces
$$\frac{d}{dt}x(t) = Ax(t) + \int_{0}^{t}B(t - s)x(s)ds +f(t, x_{t}) + h(t, x_{t})$$
This work is organized as follows : After preliminaries in the second section,  we give a main result and the conclusion.

\section{Preliminaries}
Let $X$ be a Banach Space. Firstly, we denote By $\mathbb{T}$ the group defined as the quotient $\mathbb{R}/2 \pi \mathbb{Z}$. There is an identification between functions on $\mathbb{T}$ and $2\pi$-periodic functions on $\mathbb{R}$. We consider the interval $[0, 2\pi$) as a model for $\mathbb{T}$ .\\
Given $1 \leq p < \infty $, we denote by $L^ {p}(\mathbb{T}; X)$ the space of $2\pi$-periodic locally $p$-integrable functions from $\mathbb{R}$ into $X$, with the norm:
\[\left\|f\right\|_{p}: =\left( \int_{0}^{2\pi} \left\|f(t)\right\|^{p}dt \right)^{1/p}\]
For  $f \in L^{p}(\mathbb{T}; X)$, we denote by $\hat{f}(k)$, $k \in \mathbb{Z}$ the $k$-th Fourier coefficient of $f$ that is defined by:

\[\mathcal{F}(f)(k) = \hat{f}(k) = \frac{1}{2\pi}\int_{0}^{2\pi}e^{-ikt}f(t)dt\;\; \text{for}\;\; k \in \mathbb{Z} \ \ \text{and} \  \ t \in \mathbb{R}.\]
For $1 \leq p < \infty$, the periodic vector-valued space is defined by.\\
Let S be the Schwartz space on $\mathbb{R}$ and let S' be the space of all tempered distributions on $\mathbb{R}$. Let
$\phi (\mathbb{R})$ be the set of all systems $\phi = \{\phi_{j} \}_{j \geq 0} \subset S$ satisfying
$$supp(\phi_{0}) \subset [-2, 2]$$
$$supp(\phi_{j}) \subset [-2^{j+1}, -2^{j-1}] \cup [2^{j-1}, 2^{j+1}], \  j \geq 1$$
$$\Sigma_{j \geq 0}\phi_{j}(t) = 1, \  \ t \in \mathbb{R}$$
Let $1 \leq p, q \leq \infty , s \in \mathbb{R}$ and $\phi = (\phi_{j})_{j \geq 0} \in \phi(\mathbb{R})$. We define the X-valued periodic
Besov spaces by
$$B_{p,q}^{s, \phi}((\mathbb{T}; X) = \{f \in D'(\mathbb{T}, X): ||f||_{B_{p,q}^{s, \phi}} = ( \Sigma_{j \geq 0}2^{sjq}  )^{1/q}|| \Sigma_{k \in \mathbb{Z}}e_{k}\phi_{j}(k) \hat{f}(k) ||_{p}^{q}\}$$

\begin{pr}\cite{14}\\
1) $B_{p,q}^{s}((0, 2\pi); X)$ is a Banach space;\\
2) Let $s > 0$. Then $f \in B_{p,q}^{s+1}((0, 2\pi); X)$ in and only if $f$ is differentiale  and $f' \in B_{p,q}^{s}((0, 2\pi); X)$ 
\end{pr}

\begin{defi}\cite{14}\\
For $1 \leq p < \infty$ , a sequence $\left\{M_{k}\right\}_{k \in Z} \subset B(X,Y)$ is a $B_{p,q}^{s}$-multiplier if for each $f \in B_{p,q}^{s}(\mathbb{T}, X),$ there exists  $u \in B_{p,q}^{s}(\mathbb{T},Y)$ such that $\hat{u}(k) = M_{k}\hat{f}(k)$ for all $k \in \mathbb{Z}$.
\end{defi}

\begin{defi}\cite{1}\\
The Banach space X has Fourier type $r \in ]1, 2]$ if there exists $C_{r} > 0$ such that
$$||\mathcal{F}(f)||_{r'} \leq C_{r}||f||_{r}, \  f \in L^{r}(\mathbb{R}, X)$$
where $\frac{1}{r'} + \frac{1}{r} = 1$.
\end{defi}

\begin{defi}\cite{14} \\
Let $\left\{M_{k}\right\}_{k \in Z} \subseteq B(X,Y)$  be a sequence of operators. $\left\{M_{k}\right\}_{k \in Z}$  is M-bounded of order 1( or M-bounded) if 
\begin{equation}\label{M}
\sup_{k}\|M_{k}\| < \infty \  \text{and}  \   \sup_{k}\|k(M_{k+1}-M_{k})\| < \infty
\end{equation}
\end{defi}

\begin{thm}\label{t41} \cite{1}\\
Let X and Y be Banach spaces having Fourier type $r \in ]1, 2]$ and let $\left\{M_{k}\right\}_{k \in Z} \subseteq B(X,Y)$ be a sequence satisfying (\ref{M}).
Then for $1 \leq p,q < \infty, s \in \mathbb{R}, \left\{M_{k}\right\}_{k \in Z}$ is an  $B_{p,q}^{s}$-multiplier.
\end{thm}

\begin{lem} \cite{7} \\ 
Let $L: L^{p}(\mathbb{T}, X) \rightarrow  X$  be a bounded linear operateur. Then
$$\widehat{L(u_{\textbf{.}})}(k) = L(e_{k}\hat{u}(k)):=L_{k}\hat{u}(k)\;\;  \text{for all}\;\; k \in \mathbb{Z}$$
and $\{ L_{k} \}_{k \in \mathbb{Z}}$ is r-bounded such that 
\begin{center}
$R_{p}((L_{k})_{k \in \mathbb{Z}})\leq (2r_{2\pi})^{1/p} \left\|L\right\|$.
\end{center}
\end{lem}

Next we give some preliminaries. Given $a \in L^{1}(\mathbb{R}^{+})$ and $u : [0, 2\pi] \rightarrow X$ (extended by periodicity to $\mathbb{R}$), we define
\begin{center}
$F(t) = \int_{- \infty}^{t}a(t - s)u(s)ds.$
\end{center}
Let $\tilde{a}(\lambda) =\int_{0}^{\infty} e^{- \lambda t}a(t)dt$ be the Laplace transform of $a$. An easy computation
shows that:

\begin{equation}\label{e3}
\hat{F}(k) = \tilde{a}(ik) \hat{u}(k), \text{for all} k \in \mathbb{Z}
\end{equation}

\subsection{Main result}
For convenience, we introduce the following notations: \\
$B_{k} = k A(L_{k+1} - L_{k}), P_{k}  = k(\tilde{a}(i(k+1)) - \tilde{a}(ik)), Q_{k} = k(L_{k+1} - L_{k}), R_{k} = k(G_{k+1} - G_{k}).$ 
In order to give our result, the following hypotheses are fundamental. 

\begin{enumerate}
\item[$H_{1}$]:$\{ L_{k}\}_{k \in Z}, \{ G_{k}\}_{k \in Z}$ and $\{ \tilde{a}(ik)\}_{k \in Z}$ are M-bounded (i.e $\{ P_{k}\}_{k \in Z}, \{ Q_{k}\}_{k \in Z}$ and $\{ R_{k}\}_{k \in Z}$ are bounded)
\item[$H_{2}$]: $\sup_{k \in \mathbb{Z}}\| B_{k}\|  < \infty$
\end{enumerate}

\begin{defi} :
Let $1 \leq p,q < \infty$ and $s > 0$. We say that a function $x \in  B_{p,q}^{s}(\mathbb{T}; X)$ is a strong  $B_{p,q}^{s}$-solution of (\ref{e1}) if $(x(t) - L(x_{t})) \in D(A), Dx_{t} \in B_{p,q}^{s+1}(\mathbb{T}; X)$ and equation (\ref{e1}) holds for a.e $t \in \mathbb{T}$.
\end{defi}
We prove the following result.

\begin{lem}\label{l1}:
Let X be a Banach space and A be a  linear closed operator. Suppose that $(ikD_{k} - AD_{k}-G_{k}-\tilde{a}(ik))$ is bounded invertible and $ik(ikD_{k} - AD_{k}-G_{k}-\tilde{a}(ik))^{-1}$ is bounded. 
Then  \\ $\left\{(ikD_{k} - AD_{k}-G_{k}-\tilde{a}(ik))^{-1}\right\}_{k \in Z}$ and $\left\{ik(ikD_{k} - AD_{k}-G_{k}-\tilde{a}(ik))^{-1}\right\}_{k \in Z}$ are M-bounded.
\end{lem}
\begin{proof} Let 
$S_{k} = ikN_{k}, N_{k} = (C_{k} - AD_{k} )^{-1}$ and  $C_{k} = ikD_{k} - G_{k} - \tilde{a}(ik)$.\\
For convenience, we introduce the following result
\begin{align*}
C_{k} - C_{k+1}&= [ikD_{k} - G_{k} - \tilde{a}(ik)] - [ i(k+1)D_{k+1} - G_{k+1} - \tilde{a}(i(k+1))]\\
&=ikI-ikL_{k} - G_{k} - \tilde{a}(ik)- (ik+i)(I-L_{k+1})+ G_{k+1} + \tilde{a}(i(k+1)\\
&=-iI+iL_{k+1}+ik(L_{k+1}- L_{k})+( G_{k+1}- G_{k}) + (\tilde{a}(i(k+1)-\tilde{a}(ik))
\end{align*}
Then we have
\begin{align*}
k(C_{k} - C_{k+1})&=-ikI+ikL_{k+1}+ik(k(L_{k+1}- L_{k}))+ k(G_{k+1}- G_{k}) + k(\tilde{a}(i(k+1)-\tilde{a}(ik))\\
&=-ikI + ikL_{k+1} + ikQ_{k} + R_{k} + P_{k}
\end{align*}
Now, we are going to show that

\begin{eqnarray}\label{b1}
\left\{
\begin{array}{ccccc}
\sup_{k}\|k(N_{k+1} - N_{k})\| < \infty,\\\\
\sup_{k}\|k(S_{k+1} - S_{k})\| < \infty
\end{array}
\right.
\end{eqnarray}
By hypothesis we have, $\{ N_{k}\}_{k \in \mathbb{Z}}, \{ S_{k}\}_{k \in \mathbb{Z}}, \{ T_{k} \}_{k \in \mathbb{Z}}$ and $\{ F_{k}\}_{k \in \mathbb{Z}}$ are bounded. Then We have
\begin{align*}
\sup_{k \in \mathbb{Z}}\| k (N_{k+1} - N_{k})\|&=\sup_{k \in \mathbb{Z}}\left\|(k[(C_{k+1} - AD{k+1})^{-1} - (C_{k}-  AD_{k})^{-1}]) \right\| \\
&=\sup_{k \in \mathbb{Z}}\left\|k N_{k+1}[C_{k}- AD_{k}-C_{k+1}+ AD_{k+1}]N_{k}\right\|\\
&=\sup_{k \in \mathbb{Z}}\left\| k N_{k+1}[(C_{k}-C_{k+1})- A(L_{k+1}-L_{k})]N_{k}\right\|\\ 
&=\sup_{k \in \mathbb{Z}}\left\| k N_{k+1}[k(C_{k}-C_{k+1})- kA(L_{k+1}-L_{k})]N_{k}\right\|\\ 
&=\sup_{k \in \mathbb{Z}}\left\|  N_{k+1}[-ikI + ikL_{k+1} + ikQ_{k} + R_{k} + P_{k} -  B_{k}]N_{k}\right\|\\
&=\sup_{k \in \mathbb{Z}}\left\| [-N_{k+1} + N_{k+1}L_{k+1}+ N_{k+1}Q_{k}]S_{k} + N_{k+1}R_{k}N_{k} + N_{k+1}P_{k}N_{k}-  N_{k+1} B_{k}N_{k} \right\|.
\end{align*}
We obtain: 
\begin{equation}\label{b3}
 \sup_{k \in \mathbb{Z}}\| k (N_{k+1} - N_{k})\| < \infty 
\end{equation}
On the other hand,  we have 
\begin{align*}
\sup_{k \in \mathbb{Z}}\| k (S_{k+1} - S_{k})\|&=\sup_{k \in \mathbb{Z}}\left\|k[i(k+1)(C_{k+1}-  AD_{k+1})^{-1} - ik(C_{k}- AD_{k})^{1}]\right\|\\
&=\sup_{k \in \mathbb{Z}}\left\|kN_{k+1}[i(k+1)(C_{k} - A(I-L_{k})) - ik(C_{k+1} - A(I-L_{k+1}))]N_{k}\right\| \\
&=\sup_{k \in \mathbb{Z}}\left\|kN_{k+1}[ik(C_{k}-C_{k+1})+i(C_{k} - A(I-L_{k})) - ikA(L_{k+1}-L_{k}))]N_{k}\right\| \\
&=\sup_{k \in \mathbb{Z}}\left\|N_{k+1}[k(C_{k}-C_{k+1})]ikN_{k}+ikN_{k+1} - N_{k+1} kA(L_{k+1}-L_{k}))ikN_{k}\right\| \\
&=\sup_{k \in \mathbb{Z}}\left\|N_{k+1}[k(C_{k}-C_{k+1})]ikN_{k}+ikN_{k+1} - N_{k+1} B_{k}ikN_{k}\right\| \\
&=\sup_{k \in \mathbb{Z}}\left\|N_{k+1}[-ikI + ikL_{k+1} + ikQ_{k} + R_{k} + P_{k} ]S_{k}+\frac{k}{k+1}S_{k+1} - N_{k+1} B_{k}S_{k}\right\| \\
&=\sup_{k \in \mathbb{Z}}\left\|\frac{k}{k+1}S_{k+1}[-I + L_{k+1} + Q_{k}]S_{k}+N_{k+1}[ R_{k} + P_{k} ]S_{k}+\frac{k}{k+1}S_{k+1} - N_{k+1} B_{k}S_{k}\right\| 
\end{align*}
Then  
\begin{center}$\sup_{k \in \mathbb{Z}}\| k (S_{k+1} - S_{k})\| < \infty$ \end{center}
so, $(N_{k})_{k \in \mathbb{Z}}$ and  $(S_{k})_{k \in \mathbb{Z}}$  are M-bounded. 
\end{proof}

\begin{lem}\label{l2}:
Let X be a Banach space and A be a  linear closed operator. Suppose that $(ikD_{k} - AD_{k}-G_{k}-\tilde{a}(ik))$ is bounded invertible and $ik(ikD_{k} - AD_{k}-G_{k}-\tilde{a}(ik))^{-1}$ is bounded. 
Then  $\left\{ G_{k}N_{k} \right\}_{k \in Z}$ and $\left\{ \tilde{a}(ik)N_{k} \right\}_{k \in Z}$ are M-bounded.
\end{lem}
\begin{proof} Let $T_{k} = G_{k}N_{k}$ and $F_{k} = \tilde{a}(ik)N_{k}$.\\
we have 
\begin{align*}
\sup_{k \in \mathbb{Z}}\left\|k (T_{k+1} - T_{k})\right\|&=\sup_{k \in \mathbb{Z}}\left\|(k[ G_{k+1}N_{k+1} - G_{k}N_{k}]\right\| \\
&=\sup_{k \in \mathbb{Z}}\left\|k( G_{k+1}- G_{k})N_{k+1} +  G_{k}k(N_{k+1}-N_{k}) \right\|\\
&\leq \sup_{k \in \mathbb{Z}}\left\|k( G_{k+1}- G_{k})N_{k+1}\right\| + \sup_{k \in \mathbb{Z}}\left\|G_{k}k(N_{k+1}-N_{k})\right\|\\
&\leq \sup_{k \in \mathbb{Z}}\left\|R_{k}N_{k+1}\right\| +  \sup_{k \in \mathbb{Z}}\left\|G_{k}\right\|\sup_{k \in \mathbb{Z}}\left\|k(N_{k+1}-N_{k})\right\|
\end{align*}
Then by (\ref{b3}) we have
\begin{center}$\sup_{k \in \mathbb{Z}}\| k (T_{k+1} - T_{k})\| < \infty$ \end{center}
and
\begin{align*}
\sup_{k \in \mathbb{Z}}\left\|k (F_{k+1} - F_{k})\right\|&=\sup_{k \in \mathbb{Z}}\left\|(k[ \tilde{a}(i(k+1))N_{k+1} - \tilde{a}(ik)N_{k}]\right\| \\
&=\sup_{k \in \mathbb{Z}}\left\|k( N_{k+1}- N_{k})\tilde{a}(i(k+1)) + k(\tilde{a}(i(k+1)) - \tilde{a}(ik) )N_{k}) \right\|\\
&=\sup_{k \in \mathbb{Z}}\left\|k( N_{k+1}- N_{k})\tilde{a}(i(k+1)) + F_{k}N_{k}) \right\|
\end{align*}
Then by (\ref{b3}) we have
\begin{center}$\sup_{k \in \mathbb{Z}}\| k (F_{k+1} - F_{k})\| < \infty$ \end{center}
So, $(T_{k})_{k \in \mathbb{Z}}$ and  $(F_{k})_{k \in \mathbb{Z}}$  are M-bounded.
\end{proof}

\begin{thm}
Let $1 \leq p,q < \infty$ and $s > 0$. Let X be a Banach space having Fourier type $r \in ]1, 2]$ and A be a linear operator. If  $(ikD_{k} - AD_{k}-G_{k}-\tilde{a}(ik))$ is bounded invertible and $ik(ikD_{k} - AD_{k}-G_{k}-\tilde{a}(ik))^{-1}$ is bounded. Then for every $f \in B^{s}_{p,q}(\mathbb{T}, X)$ there exist a unique strong $B^{s}_{p,q}$-solution of ({\bf \color{red}{\ref{e1}}}).
\end{thm}

\begin{proof} Define $S_{k} = ikN_{k}, N_{k} = (ikD_{k}- AD_{k}-G_{k}-\tilde{a}(ik))^{-1}$, $F_{k} = \tilde{a}(ik)N_{k}$ and $T_{k} =G_{k}N_{k}$ for $k \in \mathbb{Z}$. Since by  Lemma(\ref{l1}) and Lemma(\ref{l2}), $(S_{k})_{k \in \mathbb{Z}}, (N_{k})_{k \in \mathbb{Z}}, (F_{k})_{k \in \mathbb{Z}}$ and $(T_{k})_{k \in \mathbb{Z}}$ are M-bounded, we have  by Theorem \ref{t41} that  $(S_{k})_{k \in \mathbb{Z}}, (N_{k})_{k \in \mathbb{Z}}, (F_{k})_{k \in \mathbb{Z}}$ and $(T_{k})_{k \in \mathbb{Z}}$ are an $B_{p,q}^{s}$-multipliers. Since $D_{k}S_{k} -  AD_{k}N_{k} - T_{k} - \tilde{a}(ik)N_{k} = I$ (because   $( (ikD_{k} -AD_{k}- G_{k}- F_{k})N_{k} = I),$ we deduce $AD_{k}N_{k}$ and $D_{k}S_{k}$ are also an $B^{s}_{p,q}$-multiplicateur.\\
Now let $f \in B^{s}_{p,q}(\mathbb{T}, X)$. Then there exist $u, v, w, q, x \in  B^{s}_{p,q}(\mathbb{T}, X)$, such that\\
$\hat{u}(k) = N_{k}\hat{f}(k),  \hat{v}(k) = D_{k}S_{k}\hat{f}(k), \hat{w}(k) = T_{k}\hat{f}(k), \hat{x}(k) = F_{k}\hat{f}(k)$ and $\hat{q}(k) = AD_{k}N_{k}\hat{f}(k)$ for all $k \in \mathbb{Z}$. So,
We have $(\hat{u}(k)-L_{k}\hat{u}(k)) \in D(A)$ and $A(\hat{u}(k)-L_{k}\hat{u}(k)) = \hat{q}(k)$ for all $k \in \mathbb{Z}$, we deduce that  $(u(t)-L(u_{t})) \in D(A)$. On the other hand $\exists v \in  B^{s}_{p,q}(\mathbb{T}, X)$ such that $\hat{v}(k) = D_{k}S_{k}\hat{f}(k)=ikD_{k}N_{k}\hat{f}(k)=ikD_{k}\hat{u}(k)$.  Then we obtain $(Du_{t})' = v(t)$ a.e. Since  $Du_{t} \in B^{s+1}_{p,q}(\mathbb{T}, X)$.\\
We have $\widehat{(Du_{t})'}(k) = ikD_{k}\hat{u}(k), (A(u(.) - L(u_{.}) )^{\wedge}(k)= AD_{k}\hat{u}(k), \widehat{Gu_{.}}(k)= G_{k}\hat{u}(k)$ and\\
$\widehat{\int_{- \infty}^{t}a(t-s)u(s)ds}(k) = \tilde{a}(ik) \hat{u}(k)$ for all $k \in \mathbb{Z}$, It follows from the identity 
$$ikD_{k}N_{k} - AD_{k} N_{k} - G_{k}N_{k} - \tilde{a}(ik)N_{k} = I$$ that
\begin{center}
$(u(t)-L(u_{t}))' = A(u(t) - L(u_{t})) + G(u_{t})+ \int_{- \infty}^{t}a(t-s)u(s)ds + f(t)$
\end{center}
For the uniqueness we suppose two solutions $u_{1}$ and $u_{2}$, then  $u = u_{1} - u_{2}$ is  strong $L^{p}$-solution of equation (\ref{e1}) corresponding to the function $f = 0$, taking Fourier transform, we get $(ikD_{k} - AD_{k} - G_{k} - \tilde{a}(ik))\hat{u}(k) = 0$, which implies that $\hat{u}(k) = 0$ for all $k \in \mathbb{Z}$ and  $u(t) = 0$. Then $u_{1} = u_{2}$. The proof is completed.
\end{proof}


\section{Conclusion}
We are obtained necessary and sufficient conditions to guarantee existence and uniqueness of periodic solutions to the equation $\bf \frac{d}{dt}[x(t) - L(x_{t})]= A[x(t)- L(x_{t})]+G(x_{t})+ \int_{- \infty}^{t}a(t-s)x(s)ds+f(t)$ in terms of either the M-boundedness of the modified resolvent operator determined by the equation. Our results are obtained in the  Besov space.

\scriptsize\----------------------------------------------------------------------------------------------------------------------------------------\\\copyright
{\it{Copyright International Knowledge Press. All rights reserved.}}

\end{document}